\documentclass[11pt]{article}
\usepackage{geometry,graphicx,latexsym,amssymb,verbatim, multicol}
\usepackage{amsmath, amssymb, amsthm}
\usepackage{tikz}
\usetikzlibrary{calc}

\newtheorem{thm}{Theorem}[section]
\newtheorem{lem}[thm]{Lemma}

\newtheorem{cor}[thm]{Corollary}
\newtheorem{obs}[thm]{Observation}

\renewcommand{\comment}[1]{}

\renewenvironment{proof}{\noindent {\it Proof.}}{$\Box$\\}

\newcommand{\ov}{\overline}

\begin{document}

\begin{center}
{\LARGE  Partitions of graphs into small and large sets}
\mbox{}\\[8ex]

Asen Bojilov,
Nedyalko Nenov\footnote{This work was partially supported by the Scientific Research Fund of the St. Kliment Ohridski University of Sofia under contract No 027, 2012.}\\[1ex]
{\small Faculty of Mathematics and Informatics\\
University of Sofia\\
Bulgaria\\
bojilov@fmi.uni-sofia.bg\\
nenov@fmi.uni-sofia.bg}\\

\begin{multicols}{2}

Yair Caro\\[1ex]
{\small Dept. of Mathematics and Physics\\
University of Haifa-Oranim\\
Tivon 36006, Israel\\
yacaro@kvgeva.org.il}

\columnbreak

Adriana Hansberg\footnote{Supported by a Juan de la Cierva Postdoctoral Fellowship; Research also supported by the Ministry of Education and Science, Spain, and the European Regional Development Fund (ERDF) under project MTM2011-28800-C02-02; and under the Catalonian Government project 1298 SGR2009.}\\[1ex]
{\small Dep. de Matem\`atica Aplicada III\\
UPC Barcelona\\
08034 Barcelona, Spain\\
adriana.hansberg@upc.edu}\\[2ex]

\end{multicols}

\end{center}

\begin{abstract}
Let $G$ be a graph on $n$ vertices. We call a subset $A$ of the vertex set $V(G)$ \emph{$k$-small} if, for every vertex $v \in A$, $\deg(v) \le n - |A| + k$. A subset $B \subseteq V(G)$ is called \emph{$k$-large} if, for every vertex $u \in B$, $\deg(u) \ge |B| - k - 1$. Moreover, we denote by $\varphi_k(G)$ the minimum integer $t$ such that there is a partition of $V(G)$ into $t$ $k$-small sets, and by $\Omega_k(G)$ the minimum integer $t$ such that there is a partition of $V(G)$ into $t$ $k$-large sets. In this paper, we will show tight connections between $k$-small sets, respectively $k$-large sets, and the $k$-independence number, the clique number and the chromatic number of a graph. We shall develop greedy algorithms to compute in linear time both $\varphi_k(G)$ and $\Omega_k(G)$ and prove various sharp inequalities concerning these parameters, which we will use to obtain refinements of the Caro-Wei Theorem, the Tur\'an Theorem and the Hansen-Zheng Theorem among other things.\\

\end{abstract}

\section{Introduction and Notation}

We start with the following basic definitions. Let $n$ and $m$ be two positive integers. Let $S = \{ 0 \le a_1 \le a_2 \le \ldots \le a_m \le n-1\}$ be a sequence of $m$ integers and $\ov{S} = \{0 \le b_m \le b_{m-1} \le \ldots \le b_1 \le n-1 \}$ be the complement sequence, where $b_i = n - a_i - 1$ for $1 \le i \le m$. Let $k \ge 0$ be an integer. A subsequence $A$ of $S$ is called \emph{$k$-small} if, for every member $x$ of $A$, $x \le n - |A| + k$. A subsequence $B$ of $S$ is called \emph{$k$-large} if, for every member $x$ of $B$, $x \ge |B| - k -1$. In particular, for the terminology of graphs, we have the following definitions. Let $G$ be a graph on $n$ vertices. We call a set of vertices $A \subseteq V(G)$ \emph{$k$-small} if, for every vertex $v \in A$, $\deg(v) \le n - |A| + k$. A subset $B \subseteq V(G)$ is called \emph{$k$-large} if, for every vertex $v \in B$, $\deg(v) \ge |B| - k - 1$. When $k=0$, we say that $A$ is a \emph{small} set ($\delta$-set in \cite{Nen,BoNe}) and $B$ a \emph{large} set. Let  $S_k(G)$ denote the maximum cardinality of a $k$-small set and $L_k(G)$  denote the maximum cardinality of a $k$-large set in $G$. Further, given a graph $G$, let $\varphi_k(G)$ be the minimum integer $t$ such that there is a partition of $V(G)$ into $t$ $k$-small sets, and let $\Omega_k(G)$ be the minimum integer $t$ such that there is a partition of $V(G)$ into $t$ $k$-large sets. When $k = 0$, we will set $\varphi(G)$ instead of $\varphi_0(G)$ and $\Omega(G)$ instead of $\Omega_0(G)$. Consider the following observations.\\

\begin{obs}\label{obs1}
Let $n$ and $m$ be two positive integers. Let $S = \{ 0 \le a_1 \le a_2 \le \ldots \le a_m \le n-1\}$ be a sequence and let $G$ be a graph.\\[-5ex]
\begin{enumerate}
\item[(i)] $A$ is a $k$-small subsequence of $S$ if and only if $\ov{A}$ is a $k$-large subsequence of $\ov{S}$;
\item[(ii)] $A$ is a $k$-small set in $G$ if and only if $A$ is a $k$-large set in $\ov{G}$;
\item[(iii)] $S_k(G) = L_k(\ov{G})$ and $L_k(G) = S_k(\ov{G})$;
\item[(iv)] $\varphi_k(G) = \Omega_k(\ov{G})$ and $\Omega_k(G) = \varphi_k(\ov{G})$.
\end{enumerate}
\end{obs}

\begin{proof}
(i) $A$ is a small subsequence of $S$ if and only if $a_i \le n - |A| + k$ for every $a_i \in A$, which is equivalent to $b_i \ge |A| -k-1$ for each $b_i = n - a_i - 1  \in \ov{A}$, meaning that $\ov{A}$ is a $k$-large subsequence of $\ov{S}$.\\
(ii) $A$ is a $k$ small set of $G$ if and only if $\deg_G(v) \le n - |A| + k$ for every $v \in A$, which is equivalent to $\deg_{\ov{G}}(v) = n - 1 - \deg_G(v) \ge |A| - k -1$ for every $v \in A$, meaning that $A$ is a $k$-large set in $\ov{G}$.\\
(iii) and (iv) follow directly from (ii).
\end{proof}

 A \emph{$k$-independent set} $A$ in $G$ is a subset of vertices of $G$ such that $|N(v) \cap A| \le k$ for every $v \in A$. The maximum cardinality of a $k$-independent set is denoted by $\alpha_k(G)$. Note that a $0$-independent set is precisely an independent set, so we will use the usual notation $\alpha(G)$ for the independence number instead of $\alpha_0(G)$. The well-known Caro-Wei bound \cite{Caro, Wei-1} $\alpha(G) \ge \sum_{v \in V(G)} \frac{1}{\deg(v)+1}$ was generalized by Favaron \cite{Fav} to $\alpha_k(G) \ge \sum_{v \in V(G)} \frac{k}{1+ k \deg(v)}$. Other generalizations and improvements  were given in \cite{CaTu, Jel}. For more information on the $k$-independence number see also the survey \cite{CFHV}.

Similarly, we call a subset $B \subseteq V(G)$ such that $|N(v) \cap B| \ge |B| - k -1$ for every $v \in B$ a \emph{$k$-near clique} and the cardinality of a maximum $k$-near clique will be denoted by $\omega_k(G)$. A $0$-near clique is precisely a clique and so we will use the usual notation for the clique number $\omega(G)$ instead of $\omega_0(G)$.

 The connection between $k$-independent sets and $k$-near cliques to $k$-small and $k$-large sets is given below.

 \begin{obs}
In a graph $G$, every $k$-independent set is a $k$-small set and every $k$-near clique is a $k$-large set;
 \end{obs}

 \begin{proof}
 Let $A$ be a $k$-independent set and $B$ a $k$-near clique of $G$. Then $\deg(v) \le n-|A|+k$ for every $v \in A$ and $\deg(v) \ge |B|-k-1$ for every $v \in B$. Hence, $A$ is a $k$-small set and $B$ a $k$-large set.
 \end{proof}

We denote by $\deg(v) = \deg_G(v)$ the \emph{degree} of the vertex $v$ in $G$ and $N_G(v)$ and $N_G[v]$ is its \emph{open} and, respectively, \emph{closed neighborhood} of $v$. With {\rm d}(G) we refer to the \emph{average degree} $\frac{1}{|V(G)|}\sum_{v \in V(G)} \deg(v)$ of $G$. Given the degree sequence $d_1 \le d_2 \le \ldots \le d_n$ of $G$, we will denote by $v_1, v_2, \ldots, v_n$ the vertices of $G$ ordered accordingly to the degree sequence, i.e. such that $\deg(v_i) = d_i$. Moreover, $\chi(G)$ is the chromatic number and $\theta(G)$ the clique-partition number of $G$. For notation not mentioned here, we refer the reader to \cite{West}.

The paper is organized in several sections as follows:
\begin{enumerate}
\item[1] Introduction and Notation
\item[2] Bounds on $S_k(G)$ and $L_k(G)$ with applications to upper and lower bounds on $\alpha_k(G)$ and $\omega_k(G)$
\item[3] Algorithms for $\varphi_k(G)$ and $\Omega_k(G)$
\item[4] Bounds on $\varphi_k(G)$ and $\Omega_k(G)$
\item[5] More applications to $\alpha(G)$ and $\omega(G)$
\item[6] Variations of small and large sets
\item[7] References
 \end{enumerate}

 \section{Bounds on $S_k(G)$ and $L_k(G)$ with applications to upper bounds on $\alpha_k(G)$ and $\omega_k(G)$}

Since every $k$-independent set of $G$ is a $k$-small set and every $k$-near clique of $G$ is a $k$-large set, one expects that the bounds on $S_k(G)$, $L_k(G)$, $\varphi(G)$ and $\Omega(G)$ can be derived using their arithmetic definitions, and that some properties will be also useful in obtaining bounds on the much harder to compute $\alpha_k(G)$ and $\omega_k(G)$. As we shall see in the sequel this is indeed the case and several refinements of the Caro-Wei Theorem \cite{Caro, Wei-1}, the Tur\'an Theorem \cite{Tur} and Hansen-Zheng Theorem \cite{HaZhe} are easily derived from bounds using $k$-small sets and $k$-large sets as well as some relations between $L_0(G)$ and $\chi(G)$. A lower bound on $\alpha(G)$ and $\omega(G)$ in terms of $\Omega(G)$ and $\varphi(G)$, respectively, illustrates the usefulness of working with small and large sets.\\

\begin{thm}\label{thm_basic_ineq}
Let $G$ be a graph. Then $\alpha(G) \ge \Omega(G)$ and $\omega(G) \ge \varphi(G)$.
\end{thm}

\begin{proof}
Let $G_1 = G$ and let $x_1$ be a vertex of minimum degree in $G_1$. Now, for $i \ge 1$, let $x_i$ be a vertex of minimum degree in $G_i$ and define successively $G_{i+1} = G_i - N_{G_i}[x_i]$ and $V_i = N_{G_i}[x_i]$, until there are no vertices left, say until index $q$. In this way, we obtain a partition $V_1 \cup V_2 \cup \ldots \cup V_q$ of $V(G)$ into large sets, as, for every $v \in V_i$, $\deg_G(v) \ge \deg_{G_i}(v) \ge \deg_{G_i}(x_i) = |V_i| -1$. Hence, $q \ge \Omega(G)$. On the other side, $\{x_1, x_2, \ldots, x_q\}$ is an independent set by construction and thus $\alpha(G) \ge q$. Therefore, $\alpha(G) \ge \Omega(G)$ and also $\omega(G) = \alpha(\ov{G}) \ge \Omega(\ov{G}) = \varphi(G)$ and we are done.
\end{proof}

We mention that a more complicated proof of $\omega(G) \ge \varphi(G)$ was given in \cite{Nen}. One of the strongest lower bounds for the independence number of a graph is the so called residue of the degree sequence denoted $R(G)$ (see \cite{FMS, Tri, Jel}),  which is the number of zeros left in the end of the Havel-Hakimi algorithm. As we shall see later, computing $\Omega(G)$ requires $O(|V(G)|)$-time while the Havel-Hakimi algorithm requires $O(|E(G)|)$-time. While $R(G)$ does better than all of the lower bounds given in the survey \cite{Wil}, here are two examples showing that in one case $R(G)$ does better and in the other $\Omega(G)$ does better. For the star $G = K_{1,n}$, $R(G) = n-1$ while $\Omega(G) \sim \frac{n}{2}$. However, for the graph $G$ on $6$ vertices with degree sequence $1,2,2,3,3,3$, $\Omega(G) = 3$ while $R(G) = 2.$\\

While the above theorem gives lower bounds on $\alpha(G)$ and $\omega(G)$ in terms of $\Omega(G)$ and $\varphi(G)$, the next one gives upper bounds on $\alpha_k(G)$ and $\omega_k(G)$ in terms of $S_k(G)$ and $L_k(G)$.\\

\begin{thm}\label{thm_trivial-bounds}
Let $G$ be a graph on $n$ vertices and let $d_1 \le d_2 \le \ldots \le d_n$ its degree sequence. Then\\[-5ex]
\begin{enumerate}
\item[(i)] $S_k(G) \ge \alpha_k(G)$ and $L_k(G) \ge \omega_k(G)$;
\item[(ii)] $S_k(G) \ge \frac{n}{\varphi_k(G)}$ and $L_k(G) \ge \frac{n}{\Omega_k(G)}$;
\item[(iii)] $S_k(G)  = \max \{ s : d_s \le n - s +k \}$ and $\{v_1, v_2, \ldots, v_{S_k(G)}\}$ is a maximum $k$-small set of $G$;
\item[(iv)] $L_k(G)  = \max \{ t :  t - k - 1 \le d_{n- t +1}  \}$ and $\{v_{n-L_k+1}, v_{n-L_k+2}, \ldots, v_n\}$ is a maximum $k$-large set of $G$.
\end{enumerate}
\end{thm}

 \begin{proof}
(i) Since a  $k$-independent set is a $k$-small set and a $k$-near-clique is a $k$-large set,  $S_k(G) \ge \alpha_k(G)$ and $L_k(G) \ge \omega_k(G)$.\\
(ii) Let $V_1 \cup V_2 \cup \ldots \cup V_t$ be a partition of $V(G)$ into $t = \varphi_k(G)$ $k$-small sets. Then $S_k(G) \ge \displaystyle \max_{1 \le i \le t} |V_i| \ge \frac{n}{t} = \frac{n}{\varphi_k(G)}$. The other inequality follows from $L_k(G) = S_k(\ov{G}) \ge   \frac{n}{\varphi_k(\ov{G})} = \frac{n}{\Omega_k(G)}$.\\
(iii) Let $v_1, v_2, \ldots, v_n$ be the vertices of $G$ ordered according to its degree sequence. Let $A$ be an arbitrary $k$-small set.  Clearly, for every vertex $v \in A$, $\deg(v) \le n - |A| + k$. Now order the degrees of the vertices of $A$ in increasing order such that $\deg(u_1) \le \deg(u_2) \le \ldots \le \deg(u_{|A|}) \le n - |A| +k$. Then $d_{|A|} \le \deg(u_{|A|}) \le n - |A| + k$. Hence, for every $k$-small set $A$, $d_{|A|} \le \deg(u_{|A|}) \le n - |A| + k$. Now let $s$ be the largest index in the degree sequence of $G$ such that $d_s \le n - s +k$. Then $s \ge S_k(G)$, as this inequality holds for any $k$-small set. But observe that  $\{v_1, v_2, \ldots, v_s\}$ is $k$-small by definition. Hence $S_k(G) \ge s$ and we conclude that $s = S_k(G)$. \\
(iv) Let $\ov{d}_1 \le \ov{d}_2 \le \ldots \le \ov{d}_n$ be the degree sequence of $\ov{G}$ given through $\ov{d}_i = n - 1 - d_{n-i+1}$ for $1 \le i \le n$. Then, by item (iii) and Observation \ref{obs1}(iii) we have
$L_k(G) = S_k(\ov{G})  = \max \{ t : \ov{d}_t \le n - t +k \} = \max \{t : t - k - 1 \le d_{n- t +1}\}$ and we are done.
 \end{proof}

Note from previous theorem that if $k \ge \Delta$, then $S_k(G) = n$ and $\varphi_k(G) = 1$. Also, if $k \ge n - \delta -1$, then $L_k(G) = n$ and $\Omega_k(G) = 1$. In this sense, the restrictions $k \le \Delta$ or $k \ge n - \delta -1$ needed in some of our theorems or observations are natural.\\

From Theorem \ref{thm_trivial-bounds}, the following observation follows straightforward.\\

\begin{obs}\label{obs_SLd}
Let $G$ be a graph with minimum degree $\delta$ and maximum degree $\Delta$. Then\\[-5ex]
\begin{enumerate}
\item[(i)] $n - \Delta + k \le S_k(G) \le n - \delta + k$ for $k \le \Delta$;
\item[(ii)] $\delta + k + 1  \le L_k(G) \le \Delta + k + 1$ for $k \le n - \delta - 1$;
\item[(iii)] if $G$ is $r$-regular, then $S_k(G) = n - r + k$ when $k\le r$ and $L_k(G) = r+k+1$ when $k \le n-r-1$.
\end{enumerate}
\end{obs}

\begin{proof}
(i) Let $\delta = d_1 \le d_2 \le \ldots \le d_n = \Delta$ be the degree sequence of $G$. For $k \le \Delta$, $d_{n - \Delta+k} \le n - (n - \Delta+k) + k= \Delta$. Therefore, $n - \Delta+ k \in \{s : d_s \le n-s+k\}$ and thus $n-\Delta+k \le S_k(G)$. Moreover, according to Theorem \ref{thm_trivial-bounds}(iii), from $S_k \in \{s : d_s \le n-s +k $ it follows that $\delta \le d_{S_k(G)} \le n - S_k(G) + k$, that is, $S_k(G) \le n -\delta +k$.\\
(ii) This follows from (i) applied to the graph $\ov{G}$.\\
(iii) This follows from (i) and (ii).
\end{proof}

Next we show a connection between $L_0(G)$ and the chromatic number $\chi(G)$ strengthening $L_0(G) \ge \omega(G)$. The analogon follows for $S_0(G)$ and the clique-partition number $\theta(G)$.\\

\begin{obs}\label{obs_chi}
Let $G$ be a graph. Then \\[-5ex]
\begin{enumerate}
\item[(i)] $L_0(G) \ge \chi(G) \ge \omega(G)$;
\item[(ii)] $S_0(G) \ge \theta(G) \ge \alpha(G)$.
\end{enumerate}
\end{obs}

\begin{proof}
(i) By a result of Powell and Welsh (\cite{PoWe}, see also \cite{Har}, p. 148), $\chi(G) \le \max$ $\{\min\{i, d_i+1\} :1 \le i \le n \}$, where $d_1 \ge d_2 \ge \ldots \ge d_n$ is the degree sequence of $G$. This can be rewritten with the conventional order $d_1 \le d_2 \le \ldots \le d_n$ as $\chi(G) \le \max \{t : t \le d_{n-t+1}+1\}$. Since, by the above theorem, the last expression is equal to $L_0(G)$, we obtain, together with Theorem \ref{thm_basic_ineq}, the desired inequality chain.\\
Another proof of $L_0(G) \ge \chi(G)$ can be given the following way. Let $V_1 \cup V_2 \cup \ldots \cup V_r$ be an $r$-chromatic partition of $V(G)$, where $r = \chi(G)$. Suppose there is an index $i$ such that every vertex $v \in V_i$ has no neighbor in some set $V_j$, for an index $j \neq i$. Then we can distribute the vertices of $V_i$ among the other sets $V_j$, obtaining thus an $(r-1)$-chromatic coloring of $G$, which is a contradiction. Hence, for every $1 \le i \le r$, there is a vertex $v_i \in V_i$ such that $v_i$ has a neighbor in $V_j$ for every $1 \le j \le r$ and $i \neq j$. Therefore, $\deg(v_i) \ge r-1$ for $1 \le i \le r$ and hence $\{v_1, v_2, \ldots, v_r\}$ is a large set, yielding $\chi(G) = r \le L_0(G)$.\\
(ii) This follows from (i) and $S_0(G) = L_0(\ov{G})$, $\theta(G) = \chi(\ov{G})$ and $\alpha(G) = \omega(\ov{G})$.
\end{proof}

We close this section with three observations about partitions of the vertex set of a graph into a $k$-small set and a $k$-large set.\\

\begin{obs}\label{obs_part}
Let $G$ be a graph. Then $V(G)$ can be partitioned into a $k$-small set $V_S$  and a $k$-large set $V_ L$.
 \end{obs}

\begin{proof}
Let $d_1 \le d_2 \le \ldots \le d_n$ be the degree sequence of $G$ and let $j = S_k(G)$ be the largest index such that $d_ j  \le n - j + k$. Let $v_1, v_2, \ldots, v_n$ be the vertices of $G$ ordered according to its degree sequence. Set $V_S =  \{v_1, \ldots ,v_ j \}$ and set $V_L  = V  \setminus V_S$. Clearly, $|V_ S| = j$ and $|V_ L| = n - j$. By Theorem \ref{thm_trivial-bounds}(iii), $V_S$ is a maximum $k$-small set. Since $j$ is the maximum index for which $d_ j \le n - j +k$, it follows that $d_ { j+1} >   n - ( j +1) + k$ and thus $d _{j+1} \ge n - j + k$. But then, for $i \ge j +1$,  $d _i \ge  d_ { j +1} \ge n - j +k  = |V_L| + k > |V_L| - k - 1$, and hence $V_L$ is a $k$-large set.  Note that already a partition into small and large sets suffices to prove the statement since any small set is a $k$-small set for $k >0$ and any large set is a $k$-large set for $k > 0$.
\end{proof}

From Observation \ref{obs_part} follows, in particular, that in every $n$-vertex graph there is either a $k$-small set on at least $n/2$ vertices or a $k$-large set on at least $n/2$ vertices.\\

\begin{obs}
$n \le L_k(G) + S_k(G) \le n +1+2k$ and this is sharp.
\end{obs}

\begin{proof}
From Observation \ref{obs_part}, we obtain directly the lower bound $n \le S_k(G) + L_k(G)$. Let now $A$ be a $k$-small set realizing $S_k(G)$ and $B$ a $k$-large set realizing $L_k(G)$. If $A \cap B = \emptyset$, then clearly $|A| + |B| \le n$. Otherwise suppose there is a vertex $u \in A \cap B$. Then $\deg(u) \le n - |A| +k$ and $\deg(u) \ge |B| -k -1$. Hence $|B| - k - 1 \le n - |A| + k$ and $|A| +|B| \le n +1+2k$.\\
 To see the sharpness of the lower bound, let $G_1$ be a graph on $n_1 = 2q > 2 (2k+2)$ vertices whose vertex set can be split into an independent set $V_S$ and a clique $V_L$ with $|V_S|=|V_L| = q$, and such that their vertices are joined by $k+1$ pairwise disjoint perfect matchings. Then, the vertices in $V_S$ have all degree $k+1$ and the vertices in $V_L$ have all degree $q+k$. Hence, for the degree sequence $d_1 \le d_2 \le \ldots \le d_{2q}$ of $G_1$ we have $d_q = k+ 1 \le n_1 - q + k = q+k$ and $q+k = d_{q+1} > n_1 - (q+1) + k = q + k -1$, from which follows that $V_S$ is a maximum $k$-small set, by Theorem \ref{thm_trivial-bounds}. Also from $d_{q+1} \ge n_1 - q - k - 1 = q-k-1$ and $d_q = k+1 < n_1 - q - k - 1 = q - k - 1 $, as $q > 2k+2$, it follows by the same theorem that $V_L$ is a maximum $k$-large set of $G_1$. Hence for this graph, $n_1 = S_k + L_k$ holds. Finally, for the sharpness of the upper bound,
let $G_2$ be a graph in which the largest $2k +1$ degrees in the degree sequence are $k$. An easy check reveals the required equality.
\end{proof}

\begin{obs}
Let $G$ be a graph on n vertices and $e(G)$ edges. Then there is  partition of $V(G)$ into a $k$-small set $V_S$ and a $k$-large set $V_L$ such that $|V_L| \le \frac{1}{2} (k+1 + \sqrt{(k+1)^2 +8e(G)})$ and hence $|V_S| \ge n - \frac{1}{2}(k+1 + \sqrt{(k+1)^2 +8e(G)})$.
 \end{obs}

\begin{proof}
Let $V(G)  = V_S \cup V_L$ be a partition into a $k$-small and a $k$-large set and let $p = |V_L|$. Then
$2e(G) \ge \sum_{v \in V_ L} \deg(v) \ge p(p-k-1)$. Solving the quadratic inequality, we obtain $p \le  \frac{1}{2}(k+1 + \sqrt{(k+1)^2 +8e(G)})$.
\end{proof}

\noindent
\section{Algorithms for $\varphi_k(G)$ and $\Omega_k(G)$}

In this section, we will present two algorithms with which we will be able to calculate $\varphi_k(G)$ and $\Omega_k(G)$  for a graph $G$. For this, we consider any sequence of $m$ integers $A = \{0 \le a_1 \le \ldots \le a_m \le n-1\}$ (not necessarily graphic). Now we want to break the sequence into $k$-small subsequences. With this aim, we apply the following algorithm.

\noindent
{\bf Algorithm 1}\\[1ex]
INPUT: $A$\\
{\sc Step 1:} Set $i := 0$, $R_0 := A$.\\
{\sc Step 2:} Repeat\\[-5ex]
\begin{enumerate}
\item[(1)] $n_i := |R_i|$
\item[(2)] $p_i := \min \{n_i, n - a_{n_i}+k\}$
\item[(3)] $A_{i+1} := \{a_{n_i-p_i+1}, a_{n_i-p_i+2}, \ldots , a_{n_i}\}$
\item[(4)] $R_{i+1} := R_i \setminus A_{i+1}$
\item[(5)] $i:=i+1$
\end{enumerate}
\mbox{}\\[-5ex]
until $R_i = \emptyset$.\\[1ex]
OUTPUT: $s := i$, $A_1, A_2, \ldots, A_s$.\\

Here, $i$ stands for the current step number; $R_i$ is the set of remaining elements and $n_i$ its cardinality; $A_{i+1}$ is the new subsequence constructed in step $i$; and, on the output, $s$ it is the number of constructed subsequences $A_i$.\\

\begin{thm}
Let $A= \{0 \le a_1 \le \ldots \le a_m \le n-1\}$ be a sequence of $m$ integers. Then Algorithm 1 under input $A$ yields a minimum partition of $A$ into $s$ $k$-small subsequences $A_1, A_2, \ldots, A_s$.
\end{thm}

\begin{proof}
Clearly, $A_i$ is a subsequence of $A$ for $i = 1, 2, \ldots, s$. By construction, in each step $i$, $A_{i+1} \subseteq R_i = R_{i-1} \setminus A_i$ and so the $A_i$'s are pairwise disjoint. Moreover, the last step $s$ is attained when $R_s = \emptyset$, i.e., $A_s = R_{s-1}$, meaning that $A_s$ consists of all remaining elements of $A$. Hence, $A_1, A_2, \ldots, A_s$ is a splitting of $A$ into subsequences. We now proceed to prove that, in each step $i \ge 0$, the produced subsequence $A_{i+1}$ is $k$-small. We distinguish between the two possible situations:\\
(a) {\it $p_i = n_i \le n - a_{n_i} +k$:} Then, $A_{i+1} = \{a_1, a_2, \ldots, a_{n_i}\} = R_i$ and,  for every $a \in A_{i+1}$, we have $a \le a_{n_i} = n - (n - a_{n_i} + k) + k \le n - n_i + k = n - |A_{i+1}| + k$. Thus, $A_{i+1}$ is a $k$-small subsequence.\\
(b) {\it $p_i = n - a_{n_i}+ k \le n_i$:} Then, $A_{i+1} := \{a_{n_i-(n - a_{n_i} + k)+1}, a_{n_i-(n - a_{n_i}+k)+2}, \ldots , a_{n_i}\}$ and, for every $a \in A_{i+1}$, we have $a \le a_{n_i} = n - (n - a_{n_i} + k) + k = n - |A_{i+1}| + k$. Thus, $A_{i+1}$ is a $k$-small subsequence of $A$.\\
Finally we shall prove that the output $s$ given by Algorithm 1 is the minimum number of $k$-small subsequences in which $A$ can be partitioned. Let $A'_1, A'_2, \ldots, A'_q$ be an optimal splitting of $A$ into $k$-small sequences, i.e. such that $q$ is minimum. Then clearly $q \le s$. Let $C_i = \max \{a : a \in A'_i\}$ and, without loss of generality, assume that $C_1 \ge C_2 \ge \ldots \ge C_q$. We will show by induction on $i$ that $a_{n_i} \le C_i$. Since clearly $a_{n_1} = a_m = C_1$, the base case is done. Assume that $a_{n_i} \le C_i$ for $i = 1, 2, \ldots, r$ and an $r < q$. Then, as $A'_i$ is a $k$-small set, we have
$$n - |A_i| + k = a_{n_i} \le C_i \le n - |A'_i| + k,$$
implying that $|A'_i| \le |A_i|$, for $i = 1, 2, \ldots, r$. Suppose to the contrary that $a_{n_{r+1}} > C_{r+1}$. Then $a_{n_{r+1}} \in \cup_{i=1}^r A'_i$. As $\sum_{i=1}^r |A'_i| \le \sum_{i=1}^r |A_i|$ and, moreover, $a_{n_{r+1}} \notin \cup_{i=1}^r A_i$ by construction, there has to be an element $y$ which is contained in $\cup_{i=1}^r A_i$ but not in $\cup_{i=1}^r A'_i$. Hence, $y \in A'_j$ for some $j \ge r+1$ and $y \ge a_{n_{r+1}}$. As $C_j$ is the largest element in $A'_j$, we conclude that $C_{r+1} \ge C_j \ge y \ge a_{n_{r+1}}$, contradicting the assumption. Hence $a_{n_{r+1}} \le C_{r+1}$ and by induction it follows that $a_{n_i} \le C_i$ for all $i = 1,2, \ldots, q$. As above, this implies that $|A'_i| \le |A_i|$ for all $i = 1,2, \ldots, q$. Hence,
$$m = \sum_{i=1}^q |A'_i| \le \sum_{i=1}^q |A_i| \le \sum_{i=1}^s |A_i| = m,$$
from which we obtain $q = s$. Therefore, Algorithm 1 yields a partition of $A$ into the minimum possible number of $k$-small subsequences $A_1, A_2, \ldots, A_s$.
\end{proof}

\begin{obs}
Algorithm 1 can be written recursively by defining a function $f$ which will give the partition of an arbitrary sequence into $k$-small subsequences:\\
{\sc Step 1:} Set $f(\emptyset) = \emptyset$.\\
{\sc Step 2:} \\
$\begin{array}{ll}
f(\{0 \le a_1 \le \ldots \le a_m \le n -1\}) = &\{\{a_{m-\min\{m,n-a_m+k\}+1}, \ldots ,a_m\}\}\\
                                                                        & \cup f(\{0 \le a_1 \le \ldots \le a_{\min\{m,n-a_m+k\}} \le n-1\})
\end{array}$
\end{obs}

When $m = n$ and $d_1 \le d_2 \le \ldots \le d_n$ is the degree sequence of a graph $G$, we can use Algorithm 1 to find a partition of $V(G)$ into the minimum possible number of $k$-small sets.

\begin{cor}
Let $G$ be a graph and $d_1 \le \ldots \le d_n$ its degree sequence. Let $A = \{0 \le d_1 \le \ldots \le d_n \le n-1\}$ and let $V_1, V_2, \ldots, V_s$ be the sets of vertices corresponding to the degree subsequences $A_1, A_2, \ldots, A_s$ given by Algorithm 1 under input $A$. Then $V_1 \cup V_2 \cup \ldots V_s$ is a partition of $V(G)$ into $s = \varphi_k(G)$ $k$-small sets.
\end{cor}

By the duality between $k$-small and $k$-large sequences and since $\Omega_k(G) = \varphi_k(\ov{G})$, we can modify Algorithm 1 to an algorithm that leads us to find the exact value of $\Omega_k(G)$. Again, consider any sequence of $m$ integers $B = \{0 \le b_m \le b_{m-1} \le  \ldots \le b_1 \le n-1\}$ (not necessarily graphic).

\noindent
{\bf Algorithm 2}\\[1ex]
INPUT: $B$\\
{\sc Step 1:} Set $i := 0$, $S_0 := B$.\\
{\sc Step 2:} Repeat\\[-5ex]
\begin{enumerate}
\item[(1)] $n_i := |S_i|$
\item[(2)] $q_i := \min \{n_i, b_{n_i}+k+1 \}$
\item[(3)] $B_{i+1} := \{b_{n_i}, b_{n_i-1}, \ldots , b_{n_i - q_i+1}\}$
\item[(4)] $S_{i+1} := S_i \setminus B_{i+1}$
\item[(5)] $i:=i+1$
\end{enumerate}
\mbox{}\\[-5ex]
until $S_i = \emptyset$.\\[1ex]
OUTPUT: $t := i$, $B_1, B_2, \ldots, B_t$.\\

\begin{thm}
Let $B = \{0 \le b_m \le b_{m-1} \le \ldots \le b_1 \le n-1\}$ be a sequence of $m$ integers. Then Algorithm 2 under input $B$ yields a minimum partition of $B$ into $s$ $k$-large subsequences $B_1, B_2, \ldots, B_t$.
\end{thm}

\begin{proof}
Let $A = \ov{B} = \{0 \le a_1 \le a_2 \le \ldots \le a_m \le n-1\}$ be the complementary sequence to $B$, where $a_i = n - b_i -1$. Then, from the application of Algorithm 1 under input $A$ and of Algorithm 2 under input $B$, it follows:
\begin{enumerate}
\item[(i)] $R_0 = A = \ov{B} = \ov{S_0}$
\item[(ii)] $R_i = \ov{S_i}$, $|R_i| = n_i = |S_i|$
\item[(iii)] $q_i = \min \{n_i, b_{n_i}+k+1\} =\min \{n_i, n - (n-1-b_{n_i}) + k\} =\min\{n_i, n - a_{n_i} + k \}= p_i$
\item[(iv)] $B_{i+1} = \{b_{n_i}, b_{n_i-1}, \ldots , b_{n_i - q_i+1}\} = \{n-a_{n_i}-1, n - a_{n_i-1}-1, \ldots , n- a_{n_i - q_i+1}-1\} = \{n-a_{n_i-q_i+1}-1, \ldots, n - a_{n_i-1}-1, n- a_{n_i}-1\} = \ov{A_{i+1}}$ and
\item[(v)] $S_{i+1} = S_i \setminus B_{i+1} = \ov{R_i} \setminus A_{i+1}$.
\end{enumerate}
Moreover, $S_i = \emptyset$ if and only if $R_i = \emptyset$ and thus the number of steps performed by Algorithm 1 under input $A$ is the same as the number of steps performed by Algorithm 2 under input $B$ and hence $s = t$. Since Algorithm 1 yields a partition of $A = \ov{B}$ into the $k$-small sets $A_1, A_2, \ldots, A_s$, the output $B_1, B_2 \ldots, B_t$ of Algorithm 2 is a partition of $B$ into $k$-large sets.
\end{proof}

Again, when $m = n$ and $d_n \le d_{n-1} \le \ldots \le d_1$ is the degree sequence of a graph $G$, we can use Algorithm 2 to find a partition of $V(G)$ into the minimum possible number of $k$-large sets.

\begin{cor}
Let $G$ be a graph and $d_n \le \ldots \le d_1$ its degree sequence. Let $B = \{0 \le d_n \le \ldots \le d_1 \le n-1\}$ and let $V_1, V_2, \ldots, V_t$ be the sets of vertices corresponding to the subsequences $B_1, B_2, \ldots, B_t$ given by Algorithm 2 under input $B$. Then $V_1 \cup V_2 \cup \ldots V_t$ is a partition of $V(G)$ into $t = \Omega_k(G)$ $k$-large sets.
\end{cor}

\section{Bounds on $\varphi_k(G)$ and $\Omega_k(G)$}

\begin{thm}\label{thm_classical-Omega-phi}
Let $G$ be a graph on $n$ vertices and with average degree ${\rm d}$. Then\\[-5ex]
\begin{enumerate}
\item[(i)] $\displaystyle \varphi_k(G) \ge \sum_{v \in V(G)} \frac{1}{n-\deg(v)+k} \ge \frac{n}{n -{\rm d}+k}$;
\item[(ii)] $\displaystyle \Omega_k(G) \ge \sum_{v \in V(G)} \frac{1}{\deg(v)+k+1} \ge \frac{n}{{\rm d} + k +1}$.
\end{enumerate}
\end{thm}

\begin{proof}
(i)  Let $V_1, V_2, \ldots, V_t$ be a partition of $V(G)$ into $t = \varphi_k(G)$ $k$-small sets and set $|V_i| = n_i$, for $1 \le i \le t$. Then, as $\deg(v) \le n - n_i + k$ for each $v \in V_i$, we have
$$\sum_{v \in V(G)} \frac{1}{n - \deg(G)+k} = \sum_{i=1}^t \sum_{v \in V_i} \frac{1}{n - \deg(v) + k} \le  \sum_{i=1}^t \sum_{v \in V_i} \frac{1}{n_i} = t = \varphi_k(G)$$
Now, Jensen's inequality yields
$$\varphi_k(G) \ge \sum_{v \in V(G)} \frac{1}{n-\deg(v)+k} \ge \frac{n}{n -{\rm d}+k}.$$
(ii) Since $\Omega_k(G) = \varphi_k(\ov{G})$, we obtain from (i)
$$\Omega_k(G) = \varphi_k(\ov{G}) \ge  \sum_{v \in V(G)} \frac{1}{n-\deg_{\ov{G}}(v)+k} \ge \frac{1}{n - {\rm d}(\ov{G}) + k},$$
which is equivalent to
$$\displaystyle \Omega_k(G) \ge \sum_{v \in V(G)} \frac{1}{\deg(v)+k+1} \ge \frac{n}{{\rm d} + k +1}.$$

\end{proof}

Theorems \ref{thm_basic_ineq} and \ref{thm_classical-Omega-phi} for $k=0$ imply the following corollary.

\begin{cor}\label{cor_avdeg}
Let $G$ be a graph on $n$ vertices and average degree ${\rm d}$. Then\\[-5ex]
\begin{enumerate}
\item[(i)] $\displaystyle \omega(G) \ge \varphi(G) \ge \sum_{v \in V(G)} \frac{1}{n-\deg(v)} \ge \frac{n}{n -{\rm d}}$;
\item[(ii)] $\displaystyle \alpha(G) \ge \Omega(G) \ge \sum_{v \in V(G)} \frac{1}{\deg(v)+1} \ge \frac{n}{{\rm d} +1}$.
\end{enumerate}
\end{cor}

The first explicit proof of $\alpha(G)\ge\frac n{{\rm d}+1}$ can be found in \cite{ErGa}. Note also that item (ii) of the previous corollary improves the Caro-Wei bound $\alpha(G) \ge \sum_{v \in V(G)} \frac{1}{\deg(v) +1}$ \cite{Caro, Wei-1}. Moreover, the bound $\varphi(G) \ge \sum_{v \in V(G)} \frac{1}{n-\deg(v)}$ was given in \cite{Nen}.  From the result that $\displaystyle \alpha(G) \ge \Omega(G)$, one may ask if $\alpha_k(G) \ge (k+1) \Omega_k(G)$ holds in general. However, this is in general wrong, as can be seen by the following counter example. Let $n= (k+2) q$  for an integer $q < k+1$ and let $G =K_{1,n}$ be a star with $n$ leaves. Then, clearly, $\alpha_k(G) = n = (k+2) q$. Moreover, $\Omega_k(G) = \left\lceil \frac{n+1}{k+2} \right\rceil = q+1$, since every $k$-large set containing a vertex of degree one has cardinality at most $k+2$. Hence, in this case we have $\alpha_k(G) = (k+2) q < (k+1) (q+1) = (k+1) \Omega(G)$ for $q < k+1$. \\

In view of the above counter example the following problem seems natural. 
 
 \noindent
{\bf Problem.} Let $G$ be a graph on $n$ vertices. Is it true that $$\alpha_k(G) \ge \sum_{v \in V(G)} \frac{k+1}{\deg(v) +k+1} \ge \frac{n}{{\rm d}(G) +k+1}?$$

\begin{cor}\label{cor_edges}
Let $G$ be a graph on $n$ vertices and $e(G)$ edges. Then\\[-5ex]
\begin{enumerate}
\item[(i)] $e(G) \le \frac{1}{2} \left( n^2- \frac{n^2}{\varphi_k(G)} + nk\right)$;
\item[(ii)] $e(G) \ge \frac{1}{2} \left( \frac{n^2}{\Omega_k(G)} - n(k+1) \right)$.
\end{enumerate}
\end{cor}

\begin{proof}
(i) From Theorem \ref{thm_classical-Omega-phi} (i) and the fact that $n {\rm d} = 2 e(G)$, it follows
$\varphi_k(G) \ge \frac{n}{n -{\rm d}+k} = \frac{n^2}{n^2- 2e(G)+kn}$. Solving this inequality for $e(G)$, we obtain the desired result.\\
(ii) Similar as in (i),  from Theorem \ref{thm_classical-Omega-phi} (ii) and the fact that $n {\rm d} = 2 e(G)$, it follows that $\Omega_k(G) \ge \frac{n^2}{2 e(G) + k(n+1)}$. Solving the obtained inequality for $e(G)$, the result follows.
\end{proof}

In the special case $k=0$, Corollary \ref{cor_edges} yields $e(G) \leq \frac{n^2(\varphi(G)-1)}{2\varphi(G)}$. This bound is better than the bound $e(G)\leq\frac{n^2(\omega(G)-1)}{2\omega(G)}$ from classical Tur\'an's Theorem, because $\omega (G) \geq \varphi(G)$. To illustrate this by an example, let $G$ be the graph obtained from the graph $2K_n$ by adding $n$ new independent edges between the two copies of $K_n$. Then $\varphi(G) = 2$ and $\omega(G) = n$. From Tur\'an's Theorem we have $e(G) \leq 2n(n-1)$ and from Corollary \ref{cor_edges}(ii) follows that $e(G)\leq n^2$. The last inequality gives us the exact value of $e(G)$.\\

\begin{thm}\label{thm_upper&lower-bounds}
Let $G$ be a graph on $n$ vertices with minimum degree $\delta$, maximum degree $\Delta$ and average degree ${\rm d}$. Then:\\[-5ex]
\begin{enumerate}
\item[(i)] $\left \lceil \frac{n}{n-{\rm d}+k} \right \rceil \le \varphi_k(G)
\le \left\lceil \frac{n}{n+k-\Delta} \right \rceil$;
\item[(ii)] $\left \lceil \frac{n}{{\rm d} + k + 1} \right \rceil \le \Omega_k(G) \le \left \lceil \frac{n}{\delta + k + 1} \right \rceil$;
\item[(iii)] If $\frac{r-2}{r-1}n + k < {\rm d} \le \Delta \le\frac{r-1}{r}n+k$, then $\varphi_k(G) = r$; 
\item[(iv)] If $\frac{n}{r} - k - 1 \le \delta \le {\rm d} < \frac{n}{r-1} - k - 1$, then $\Omega_k(G) = r$;
\item[(v)] If $G$ is $r$-regular, then $\varphi_k(G) = \left \lceil \frac{n}{n+k-r} \right \rceil$ and $\Omega_k(G) = \left \lceil \frac{n}{r + k + 1} \right \rceil$.
\end{enumerate}
\end{thm}

\begin{proof}
(i) From Theorem \ref{thm_classical-Omega-phi}(i), it follows directly
$$\varphi_k(G) \ge \left \lceil \frac{n}{n-{\rm d}+k} \right \rceil.$$
Let now $G$ be a graph on $n$ vertices and with maximum degree $\Delta$. If $k > \Delta$, then $\varphi_k(G) =1$ and the right inequality side is obvious. So let $k \le \Delta$ and let $A \subseteq V(G)$ be a set of cardinality $n - \Delta +k$. Then, for any $v \in A$, $\deg(v) \le \Delta = n - (n - \Delta+k) +k = n - |A| +k$ and hence $A$ is a $k$-small set. Now we will partition $V(G) \setminus A$ into $k$-small sets. Note that $|V(G) \setminus A| = \Delta - k$. So take a partition $V_1, V_2, \ldots, V_t$ of $V(G) \setminus A$ into $t = \left\lceil \frac{\Delta-k}{n-\Delta+k} \right \rceil$ sets such that $|V_i| = n-\Delta +k$ for $i = 1, 2, \ldots, t-1$ and $|V_t| \le n- \Delta + k$. Since, for every vertex $v \in V_i$, $\deg(v) \le \Delta = n- (n - \Delta +k) + k \le n - |V_i|+k$, $V_i$ is a $k$-small set, for $1 \le i \le t$. Hence $A \cup V_1 \cup V_2 \cup \ldots \cup V_t$ is a partition of $V(G)$ into $1 + t = 1 + \left\lceil \frac{\Delta-k}{n-\Delta+k}\right \rceil = \left\lceil \frac{n}{n - \Delta+k}\right \rceil$ $k$-small sets, and thus
$$\varphi_k(G) \le \left\lceil \frac{n}{n - \Delta +k}\right \rceil.$$
(ii) Theorem \ref{thm_classical-Omega-phi}(ii) yields
$$\Omega_k(G) \ge \left \lceil \frac{n}{{\rm d} + k + 1} \right \rceil.$$
The other inequality side is obtained from (i) through $\Omega_k(G) = \varphi_k(\ov{G}) \le \left\lceil \frac{n}{n+k-\Delta(\ov{G})} \right \rceil = \left\lceil \frac{n}{\delta + k +1} \right \rceil.$\\
(iii) If $\frac{r-2}{r-1}n + k < {\rm d} \le \Delta \le\frac{r-1}{r}n+k$, we obtain from (i)
$$\displaystyle r-1 = \left\lceil \frac{n}{n- \frac{r-2}{r-1}n} \right\rceil  <  \left\lceil \frac{n}{n + k- {\rm d}} \right\rceil \le \varphi_k(G)
\le \left\lceil \frac{n}{n +k- \Delta} \right\rceil \le  \left\lceil \frac{n}{n- \frac{r-1}{r}n} \right\rceil = r$$
and thus $\varphi_k(G) = r$.\\
(iv) If $\frac{n}{r} - k - 1 \le \delta \le {\rm d} < \frac{n}{r-1} - k - 1$, we obtain from (ii)
$$ r-1= \left \lceil \frac{n}{\frac{n}{r-1}} \right \rceil  < \left \lceil \frac{n}{{\rm d} + k + 1} \right \rceil \le \Omega_k(G) \le \left \lceil \frac{n}{\delta + k + 1} \right \rceil \le \left \lceil \frac{n}{\frac{n}{r}} \right \rceil = r$$
and thus $\Omega_k(G) = r$.\\
(v) Recall from (i) that $\left\lceil \frac{n}{n +k- {\rm d}} \right\rceil \le \varphi_k(G)
\le \left\lceil \frac{n}{n +k- \Delta} \right\rceil$ and thus, if ${\rm d} = \Delta = r$, we have $\varphi_k(G) = \left \lceil \frac{n}{n+k-r} \right \rceil$. Analogously, item (ii) yields $\Omega_k(G) =  \left \lceil \frac{n}{r + k + 1} \right \rceil$.
\end{proof}

\section{More applications to $\alpha(G)$ and $\omega(G)$}

\begin{thm}\label{thm_Hansen}
Let $G$ be a graph on $n$ vertices and with minimum degree $\delta$ and maximum degree $\Delta$. Then\\[-5ex]
\begin{enumerate}
\item[(i)] $\alpha_k(G) \le S_k(G) \le \frac{n- \Delta + k}{2} + \sqrt{\frac{(n- \Delta +k)^2}{4} + n\Delta - 2 e(G)}$;
\item[(ii)] $\omega_k(G) \le L_k(G) \le \frac{\delta + k + 1}{2} + \sqrt{\frac{(\delta +k+1)^2}{4} - n\delta + 2 e(G)}$.\end{enumerate}
Moreover, all bounds are sharp for regular graphs.
\end{thm}

\begin{proof}
(i) Let $A$ be a maximum $k$-small set and let $\Delta$ be the maximum degree of $G$. Then $\deg(v) \le n - |A| + k$ for all $v \in A$. Then
\begin{eqnarray*}
2 e(G) = \sum_{v \in V(G)} \deg(v) &=& \sum_{v \in A} \deg(v) + \sum_{v \in V(G) \setminus A} \deg(v)\\
                                                             &\le& (n - |A| + k)|A| + \Delta (n-|A|)\\
                                                             &=&-|A|^2 + (n - \Delta +k) |A| + n \Delta,
\end{eqnarray*}
 which implies that $|A|^2 - (n-\Delta+k)|A| - n\Delta + 2 e(G)\le 0$. Solving the quadratic inequality, we obtain the desired bound
$$\alpha_k(G) \le S_k(G) \le \frac{n- \Delta + k}{2} + \sqrt{\frac{(n- \Delta +k)^2}{4} + n\Delta - 2 e(G)}.$$
Finally, if $G$ is $r$-regular, by Observation \ref{obs_SLd}(iii), all inequalities become equalities.\\
(ii) This follows from $\omega_k(G) = \alpha_k(\ov{G})$ and item (i).
\end{proof}

The following corollary is straightforward from previous theorem and Observation \ref{obs_chi}.\\

\begin{cor}\label{cor_Hansen}
Let $G$ be a graph on $n$ vertices, with maximum degree $\Delta$ and minimum degree $\delta$. Then\\[-5ex]
\begin{enumerate}
\item[(i)] $\alpha(G) \le \theta(G)\le S_0(G)
                                     \le \left\lfloor \frac{n-\Delta}{2} + \sqrt{\frac{(n-\Delta)^2}{4} + n\Delta - 2 e(G)} \right\rfloor
                                     \le \left\lfloor \frac{1}{2} + \sqrt{ \frac{1}{4} + n^2 - n - 2e(G)} \right\rfloor$;
\item[(ii)] $\omega(G) \le\chi(G) \le L_0(G)
                                       \le \left\lfloor \frac{\delta + 1}{2} + \sqrt{\frac{(\delta+1)^2}{4} - n\delta + 2 e(G)} \right\rfloor
                                        \le \left\lfloor \frac{1}{2} + \sqrt{\frac{1}{4} + 2 e(G)}\right\rfloor$.
\end{enumerate}
\end{cor}

\begin{proof}
(i) This follows from Observation \ref{obs_chi}(i) and Theorem \ref{thm_Hansen}(i) setting $k=0$. The last inequality follows because the expression is monotone increasing with $\Delta$ and $\Delta \le n-1$.\\
(ii) This follows from (i), Observation \ref{obs_chi}(ii) and $L_0(G) = S_0(\ov{G})$. \end{proof}

Note that item (i) of Corollary \ref{cor_Hansen} is a refinement of the Hansen-Zheng bound \cite{HaZhe} which states that $\alpha(G) \le \left\lfloor{\frac{1}{2} + \sqrt{\frac{1}{4} + n^2 - n - 2 e(G)}} \right\rfloor$. The inequality $\chi(G)\le\left\lfloor \frac{1}{2} + \sqrt{\frac{1}{4} + 2 e(G)}\right\rfloor$ also is well known (cf. Proposition 5.2.1 in \cite{Die}).\\



We will need the following notation. For a set $A$ of vertices of a graph $G$, let ${\rm {\rm d_r}}(A) = \sqrt[r]{\frac{1}{|A|} \sum_{v \in A}\deg^r(v)}$. When $r= 1$, we will set ${\rm d}(A)$ for ${\rm d_1}(A)$ and when $A = V(G)$, we will set ${\rm {\rm d_r}}(G)$ instead of ${\rm {\rm d_r}}(V(G))$. Note that ${\rm d}(G)$ is the average degree of $G$. In the following, we will show that the inequality $\varphi(G)\ge\frac n{n-{\rm d}(G)}$ given in Corollary \ref{cor_avdeg} can be improved when  $\rm d(G)$ is substituted by $\rm d_3(G)$. However, we will also show that, for $r \ge 4$, ${\rm d}(G)$ will not be able to be replaced by ${\rm d}_r(G)$ that easily. First, we need to prove the following lemma.\\

\begin{lem}\label{lem_tool}
Let $\beta_1, \beta_2, \ldots, \beta_r \in [0,1]$ be real numbers such that $\beta_1 + \beta_2 + \ldots + \beta_r \le r-1$. Then
\begin{equation}\label{ineq}
\sum_{i = 1}^r (1-\beta_i) \beta_i^r \le \left(\frac{r-1}{r} \right)^r
\end{equation}
and equality holds if and only if $\beta_1 = \beta_2 = \ldots = \beta_r = \frac{r-1}{r}$.
\end{lem}

\begin{proof}
If $r=1$, then $\beta_1=0$ and the inequality is obvious. Let $r \geq2$. We consider the function
$f(x)=(1-x)x^{r-1}$, $x\geq0$. From $f'(x)=x^{r-2}\big((r-1)-rx\big)$ we see that $f(x)$ attains its absolute maximum exactly when $x = \frac{r-1}{r}$ and thus
\begin{equation*}
f(x) \leq f\left(\frac{r-1}r\right)= \frac1r\left(\frac{r-1}r\right)^{r-1}.
\end{equation*}
Hence, we have
$$
(1-\beta_i)\beta_i^r=(1-\beta_i)\beta_i^{r-1}\beta_i\leq
\frac1r\left(\frac{r-1}r\right)^{r-1}\beta_i,\quad i=1,2,\dots,r.
$$
Now the condition $\beta_1+\beta_2+\cdots+\beta_r\leq r-1$ yields
$$
\sum_{i=1}^r(1-\beta_i)\beta_i^r
\leq\frac1r\left(\frac{r-1}r\right)^{r-1}(\beta_1+\beta_2+\cdots+\beta_r)
\leq\left(\frac{r-1}r\right)^r
$$
and the desired inequality holds.
Suppose now that we have equality in \eqref{ineq}. Then we have equality in all the above given inequalities and hence
$$
(1-\beta_i)\beta_i^{r-1}=\frac1r\left(\frac{r-1}r\right)^{r-1},\quad
i=1,2,\dots,r,
$$
implying thus
$\beta_1=\beta_2=\dots=\beta_r=\dfrac{r-1}r.$
\end{proof}

\begin{thm}\label{thm_d_r}
Let $G$ be a graph on $n$ vertices. Then, the following statements hold:\\[-5ex]
\begin{enumerate}
\item[(i)] For every integer $r \le \varphi(G)$, $\varphi(G) \ge \frac{n}{n-{\rm d_r}(G)}.$ Moreover, equality holds if and only if $G$ is an $\frac{n(\varphi(G)-1)} {\varphi(G)}$-regular graph.
\item[(ii)] $\varphi(G) \ge \frac{n}{n-{\rm d_3}(G)}$.
Moreover, equality holds if and only if $G$ is an $\frac{n(\varphi(G)-1)} {\varphi(G)}$-regular graph.
\item[(iii)] If $\varphi(G) \neq 2$, then $\varphi(G) \geq \frac n{n-{\rm d_4}(G)}$. Moreover, there exists a graph $G$  for which $\varphi(G)=2$ and $\varphi(G) < \frac n{n-{\rm d_4}(G)}$.
\end{enumerate}
\end{thm}

\begin{proof}
(i) Since ${\rm d_{r-1}}(G) \le {\rm d_r}(G)$ for all $r \le \varphi(G)$, it is enough to prove $\varphi(G) \ge \frac{n}{n-{\rm d}_{\varphi(G)}(G)}$. Let $\varphi(G) = \varphi$ and let $V(G) = V_1 \cup V_2 \cup \ldots \cup V_{\varphi}$ be a partition of $V(G)$ into small sets and let $n_i = |V_i|$, $1 \le i \le \varphi$. As $\deg(v) \le n - n_i$ for every $v \in V_i$ and $1 \le i \le \varphi$, we have
\begin{equation}\label{ineq2.1}
({\rm d_{\varphi}}(G))^{\varphi} n = \sum_{v \in V(G)} \deg^{\varphi}(v) = \sum_{i = 1}^{\varphi} \sum_{v \in V_i} \deg^{\varphi}(v) \le \sum_{i=1}^{\varphi} n_i (n-n_i)^{\varphi}.
\end{equation}
Setting $\beta_i = 1 - \frac{n_i}{n}$ for $1 \le i \le \varphi$, the inequality above can be rewritten as
\begin{equation}\label{ineq2.2}
({\rm d_\varphi}(G))^{\varphi} n = \sum_{v \in V(G)} \deg^{\varphi}(v) \le n^{\varphi+1} \sum_{i = 1}^{\varphi} (1 - \beta_i) \beta_i^{\varphi}.
\end{equation}
Since $\beta_1 + \beta_2 + \ldots + \beta_{\varphi} = {\varphi}-1$, Lemma \ref{lem_tool} yields ${\rm d_{\varphi}}(G) \le \frac{n ({\varphi}-1)}{\varphi}$, from which follows the desired inequality $\varphi(G) = \varphi \ge \frac{n}{n-{\rm d_{\varphi}}(G)}$. Hence we have proved
\begin{equation}\label{ineq2.3}
\varphi \ge \frac{n}{n-{\rm d}_{\varphi}} \ge \frac{n}{n-{\rm d}_{\varphi-1}(G)} \ge \ldots \ge \frac{n}{n-{\rm d}_{r}(G)}.
\end{equation}
for any $1 \le r \le \varphi(G)$.

Suppose now that we have $\varphi(G) = \frac{n}{n-{\rm d_r}(G)}$ for some $1 \le r \le \varphi = \varphi(G)$. Then, we have equality allover the inequality chain \eqref{ineq2.3}. In particular, $\varphi = \frac{n}{n-{\rm d}_{\varphi}(G)}$, which is equivalent to ${\rm d}_{\varphi} = \frac{n(\varphi -1)}{\varphi}$, and hence we have equality in \eqref{ineq2.1} and \eqref{ineq2.2}, too. From the equality in \eqref{ineq2.1}, it follows $\deg(v) = n- n_i$ for $v \in V_i$, $1 \le i \le \varphi$. From ${\rm d}_r=\frac{n(\varphi-1)}{\varphi}$ and the equality in \eqref{ineq2.2}, we see that in \eqref{ineq} there is equality, too.
Moreover, from Lemma \ref{lem_tool} it follows (with $r = \varphi$) that $\beta_i = \frac{\varphi-1}{\varphi}$ for $1 \le i \le {\varphi}$ and thus $n_ i = \frac{n}{\varphi}$ and $\varphi$ divides $n$. Hence, $\deg(v) = n -n_i = \frac{n({\varphi}-1)}{\varphi}$ for all $v \in V_i$ and $1 \le i \le {\varphi}$, turning out that $G$ is $\frac{n({\varphi}-1)}{\varphi}$-regular. Conversely, if $G$ is $\frac{n({\varphi}-1)}{\varphi}$-regular, then evidently ${\rm d_{\varphi}}(G) = \frac{n(\varphi-1)}{\varphi} = {\rm d_r}(G)$ for every $r \le \varphi$. Then from Theorem \ref{thm_upper&lower-bounds} we have $\varphi(G) = \left\lceil \frac{n}{n - \frac{n({\varphi}-1)}{\varphi}} \right\rceil = \frac{n}{n - \frac{n({\varphi}-1)}{\varphi}} = \frac{n}{n-{\rm d_r}(G)}$.\\
(ii) If $\varphi = \varphi(G) \ge 3$, then from item (i) we have $\varphi(G) \ge \frac{n}{n-{\rm d_3}(G)}$ with equality if and only if $G$ is $\frac{(\varphi-1)n}{\varphi}$-regular. It remains to consider the cases $\varphi(G) = 1$ and $\varphi(G) = 2$. Note that $\varphi(G) = 1$ holds if and only if $G = \ov{K_n}$. Hence, in this case ${\rm d_3}(G) = 0$ and $\varphi(G) =1 = \frac{n}{n-{\rm d_3}(G)}$. So assume that $\varphi(G) = 2$ and let $V(G) = V_1 \cup V_2$ be a partition of $V(G)$ into two small sets. Setting $|V_1| = n_1$ and $|V_2| = n_2 = n- n_1$, we have
$$\sum_{v \in V(G)} \deg^3(v) = \sum_{v \in V_1} \deg^3(v) + \sum_{v \in V_2} \deg^3(v) \le n_1(n-n_1)^3 + n_2(n-n_2)^3 = n_1n_2(n^2 - 2n_1n_2).$$
The last expression takes its maximum when $n_1n_2 = \frac{n^2}{4}$. Hence, it follows $\sum_{v \in V(G)} \deg^3(v) \le \frac{n^4}{8}$ and thus ${\rm d_3}(G) \le \frac{n}{2}$, which yields $\frac{n}{n- {\rm d_3}(G)} \le 2 = \varphi(G)$.\\
Now suppose that $\varphi(G) = 2 = \frac{n}{n-{\rm d_3}(G)}$. Then we have equality in the inequality given above. Hence, $n_1n_2 = \frac{n^2}{4}$ and $\deg(v) = n - n_i$ for $v \in V_1$, $i = 1, 2$. Therefore, $n_1 = n_2 = \frac{n}{2} = \frac{n (\varphi-1)}{\varphi}$ and $G$ is an $\frac{n (\varphi-1)}{\varphi}$-regular graph. On the other side, if $G$ is an $\frac{n}{2}$-regular graph, then ${\rm d_3}(G) = \frac{n}{2}$ and, from Theorem \ref{thm_upper&lower-bounds} (v), $\varphi(G) = 2$. Hence $\varphi(G) =  2 = \frac{n}{n - {\rm d_3}(G)}$.\\
(iii) The case $\varphi(G)=1$ is trivial. If $\varphi(G) \geq 4$, then
the statement follows from item (i). The case $\varphi(G) = 3$ can be proved by straightforward calculations using
Lagrange multipliers. As in the case (i), a partition of $V(G)$ into $\varphi(G) =3$ small sets $V_1, V_2, V_3$ with $|V_1| = n_1$, $|V_2|=n_2$ and $|V_3| = n_3$ leads to the inequality
\begin{equation*}
\left(\frac{{\rm d}_4(G)}{n} \right)^4 \le \sum_{i=1}^3(1-\beta_i)\beta_i^4 = f(\beta_1,\beta_2,\beta_3),
\end{equation*}
where $\beta_i = 1 - \frac{n_i}{n}$ and clearly $\beta_1+\beta_2+\beta_3=2$ and $\beta_i\in[0,1]$, for $i=1,2,3$. We will show that $ f(\beta_1,\beta_2,\beta_3) \le \left( \frac{2}{3}\right)^4$. Let
\begin{equation*}
F(\beta_1,\beta_2,\beta_3,\lambda)=\sum_{i=1}^3(1-\beta_i)\beta_i^4+\lambda(\beta_1+\beta_2+\beta_3-2)
\end{equation*}
be the Lagrange function. The extremal points are either solutions of the system
\begin{equation*}
\left|
\begin{array}{l}
\frac{\partial F}{\partial \beta_i}=4\beta_i^3-5\beta_i^4-\lambda=0,\quad i=1,2,3\\
\frac{\partial F}{\partial \lambda}=\beta_1+\beta_2+\beta_3-2=0\\
\end{array}
\right.
\end{equation*}
or they are points on the border.
We shall prove that the system has no solution in which $\beta_1,\beta_2,\beta_3$ are pairwise distinct. Let us suppose the contrary. Then $\beta_1,\beta_2,\beta_3$ are roots of
$g(x)=5x^4-4x^3+\lambda$. As $\beta_1+\beta_2+\beta_3=2$ from Vieta's formula follows that the fourth root of $g$ is $-\frac65$. Therefore $\lambda=-12\left(\frac65\right)^2$ and so $g(x)$ has only two real roots, which is a contradiction. Let $(\beta_1,\beta_2,\beta_3)$ be an extremal point which is not on the border. As $\beta_1,\beta_2,\beta_3$ are solutions of the system, we can suppose that $\beta_1=2\beta$ and $\beta_2=\beta_3=1-\beta$, where $\beta \in [0,\frac12]$. Then
\begin{equation*}
f(\beta_1,\beta_2,\beta_3)=f(\beta)=-30\beta^5+8\beta^4+12\beta^3-8\beta^2+2\beta
\end{equation*}
and
\begin{equation*}
f'(\beta)=-2(3\beta-1)(25\beta^3+3\beta^2-5\beta+1).
\end{equation*}
$f'$ has two real roots, $\frac13$ and another one negative. Therefore, $f$ attains its maximum $\left(\frac23\right)^4$ in $[0,\frac12]$ exactly when $\beta=\frac13$. It is easy to see that the maximum on the border is $\frac1{12}$, 
which is strictly smaller than $\left(\frac23\right)^4$. Hence, we have
$\left(\frac{{\rm d}_4(G)}{n}\right)^4 \le \left(\frac{2}{3}\right)^4 = \left( \frac{\varphi(G)-1}{\varphi(G)} \right),$
implying thus that $\varphi(G) \ge \frac{n}{n- {\rm d}_4(G)} $.\\
Consider now the graph $G=K_{1,9}$. It is clear that
$\varphi(G)=2$, ${\rm d_4}(G)=\sqrt[4]{657}>5$. Therefore $2 = \varphi(G)<\frac{10}{10-{\rm d_4}(G)}$.
\end{proof}

\begin{cor}\label{cor_d_r}
Let $G$ be a graph on $n$ vertices. Then, the following statements hold:\\[-5ex]
\begin{enumerate}
\item[(i)] For every integer $r \le \varphi(G)$, $\omega(G) \ge \frac{n}{n-{\rm d_r}(G)}$ and equality holds if and only if $G$ is a complete $\omega(G)$-partite Tur\'an graph $K_{\frac{n}{\omega(G)}, \frac{n}{\omega(G)}, \ldots, \frac{n}{\omega(G)}}$.
\item[(ii)] $\omega(G) \ge \frac{n}{n-{\rm d_3}(G)}$ and equality holds if and only if $G$ is a complete  $\omega(G)$-partite Tur\'an graph $K_{\frac{n}{\omega(G)}, \frac{n}{\omega(G)}, \ldots, \frac{n}{\omega(G)}}$.
\item[(iii)] If $\varphi(G) \neq 2$, then $\omega(G) \geq \frac n{n-{\rm d_4}(G)}$.
\end{enumerate}
\end{cor}

\begin{proof}
(i) From Theorems \ref{thm_basic_ineq} and \ref{thm_d_r}(i), we have
$\omega(G)\geq \varphi(G) \ge \frac n{n-{\rm d_r}(G)}$. Suppose now that $\omega(G) = \frac n{n-{\rm d_r}(G)}$.
Then we have equality in Theorem \ref{thm_d_r}(i). Thus, setting $\varphi(G) = \omega(G) = \omega$, $G$ is $\frac{n(\omega-1)}{\omega}$-regular and $e(G) = \frac{n^2(\omega-1)}{2\omega}$. Since $\omega(G) = \omega$, from Tur\'an's Theorem it follows that $G$ is a complete $\omega$-chromatic regular graph, i.e. $G$ is a complete $\omega$-partite Tur\'an graph $K_{\frac{n}{\omega}, \frac{n}{\omega},\ldots, \frac{n}{\omega}}$. Conversely, if $G$ is the complete $\omega$-partite Tur\'an graph $K_{\frac{n}{\omega}, \frac{n}{\omega},\ldots, \frac{n}{\omega}}$, then evidently $d_r(G) = \frac{n(\omega-1)}{\omega}$  and hence $\omega(G) = \omega=  \frac{n}{n - {\rm d_r}(G)}$.\\
(ii) From Theorems \ref{thm_basic_ineq} and \ref{thm_d_r}(ii), we have
$\omega(G)\geq \varphi(G) \ge \frac n{n-{\rm d_3}(G)}$. Suppose now that $\omega = \omega(G) = \frac n{n-{\rm d_3}(G)}$.
Then $\varphi(G) = \frac n{n-{\rm d_3}(G)}$, i.\,e. we have equality in Theorem \ref{thm_d_r}(ii). Thus, setting $\varphi(G) = \omega(G) = \omega$, $G$ is $\frac{n(\omega-1)}{\omega}$-regular and $e(G) = \frac{n^2(\omega-1)}{2\omega}$. Since $\omega(G) = \omega$, from Tur\'an's Theorem it
follows that $G$ is a complete $\omega$-chromatic regular graph, i.e. $G = K_{\frac{n}{\omega}, \frac{n}{\omega},\ldots, \frac{n}{\omega}}$. Conversely, if $G$ is the complete $\omega$-partite Tur\'an graph $K_{\frac{n}{\omega}, \frac{n}{\omega},\ldots, \frac{n}{\omega}}$, then evidently $d_3(G) = \frac{n(\omega-1)}{\omega}$ and hence $\omega(G) = \omega=  \frac{n}{n - {\rm d_3}(G)}$.\\
(iii) This follows from Theorems \ref{thm_basic_ineq} and \ref{thm_d_r}(iii).
\end{proof}

Note that Theorem \ref{thm_d_r}(ii) improves the bound $\varphi(G) \ge \frac{n}{n - d_2(G)}$ given in \cite{BoNe} and Corollary \ref{cor_d_r}(ii) is better than the inequality $\omega(G) \ge \frac{n}{n - d_2(G)}$, given in \cite{EdEl} and later in \cite{BoNe} where the proof was corrected.\\

Since $\alpha(G) = \omega(\ov{G})$ and $\Omega(G) = \varphi(\ov{G})$, we have the following corollaries.

\begin{cor}
Let $G$ be a graph on $n$ vertices. Then, the following statements hold:\\[-5ex]
\begin{enumerate}
\item[(i)] For every integer $r \le \Omega(G)$, $\Omega(G) \ge \frac{n}{n-{\rm d_r}(\ov{G})}.$ Moreover, equality holds if and only if  $G$ is an $(\frac{n}{\Omega(G)}-1)$-regular graph.
\item[(ii)] $\Omega(G) \ge \frac{n}{n-d_3(\ov{G})}$.
Moreover, equality holds if and only if $G$ is an $(\frac{n}{ \Omega(G)}-1)$-regular graph.
\item[(iii)] If $\Omega(G) \neq 2$, then $\Omega(G) \geq \frac n{n-d_4(\ov{G})}$. Moreover, there exists a graph $G$  for which $\varphi(G)=2$ and $\Omega(G) < \frac n{n-d_4(\ov{G})}$.\\
\end{enumerate}
\end{cor}

\begin{cor}
Let $G$ be a graph on $n$ vertices. Then, the following statements hold:\\[-5ex]
\begin{enumerate}
\item[(i)] For every integer $r \le \Omega(G)$, $\alpha(G) \ge \frac{n}{n-{\rm d_r}(\ov{G})}$ and equality holds if and only if $G$ is the union of $\alpha(G)$ copies of $K_{\frac{n}{\alpha(G)}}$.
\item[(ii)] $\alpha(G) \ge \frac{n}{n-d_3(\ov{G})}$ and equality holds if and only if $\alpha(G)$ copies of $K_{\frac{n}{\alpha(G)}}$.
\item[(iii)] If $\Omega(G) \neq 2$, then $\alpha(G) \geq \frac n{n-d_4(\ov{G})}$.
\end{enumerate}
\end{cor}

\section{Variations of small and large sets}

Let $G$ be a graph on $n$ vertices and $A$ a subset of $V(G)$.  We call $A$ \emph{$\alpha$-small} if $\sum_{v \in A} \frac{1}{n-\deg(v)} \le 1$ and \emph{$\beta$-small} if ${\rm d}(A) \le n - |A|$. Now we observe the following.

\begin{obs}\label{obs_alpha-beta}
In a graph $G$, every small set is an $\alpha$-small set and every $\alpha$-small set is a $\beta$-small set.
\end{obs}

\begin{proof}
If $A$ is a small set of $G$, then $n - \deg(v) \ge |A|$ for every vertex $v \in A$ and we have $\sum_{v \in A}\frac{1}{n - \deg(v)} \le \sum_{v \in A} \frac{1}{|A|} = 1$. Hence, $A$ is an $\alpha$-small set. Further, if $A$ is an $\alpha$-small set of $G$, then $1 \ge \sum_{v \in A} \frac{1}{n -\deg(v)} \ge  \frac{|A|}{n - {\rm d}(A)}$ by Jensen's inequality and hence ${\rm d}(A) \le n- |A|$ and thus $A$ is a $\beta$-small set.
\end{proof}

Let now $\varphi^{\alpha}(G)$ and $\varphi^{\beta}(G)$ be the minimum number of $\alpha$-small sets and, respectively, $\beta$-small sets in which $V(G)$ can be partitioned. Further, let $CW(G) = \sum_{v \in V(G)} \frac{1}{\deg(v) +1}$ be the Caro-Wei bound.\\

 \begin{thm}\label{thm_ineq-chains}
 Let $G$ be a graph on $n$ vertices. Then\\[-5ex]
 \begin{enumerate}
 \item[(i)] $\omega(G) \ge \varphi(G) \ge \varphi^{\alpha}(G) \ge \varphi^{\beta}(G) \ge \left\lceil \frac{n}{n- {\rm d(G)}} \right\rceil$;
 \item[(ii)] $\omega(G) \ge \varphi(G) \ge \varphi^{\alpha}(G) \ge CW(\ov{G})  \ge \left\lceil \frac{n}{n- {\rm d(G)}}\right\rceil$.
  \end{enumerate}
 \end{thm}

 \begin{proof}
Since every small set is an $\alpha$-small set and every $\alpha$-small set is a $\beta$-small set and because of Theorem \ref{thm_basic_ineq}, we have the inequality chain $\omega(G) \ge \varphi(G) \ge \varphi^{\alpha}(G) \ge \varphi^{\beta}(G)$. Now we will prove the remaining bounds.\\
(i) Let $t = \varphi^{\beta}(G)$ and let $V(G) = A_1 \cup A_2 \cup \ldots \cup A_t$ be a partition of $V(G)$ into $\beta$-small sets. Then, using the definition of $\beta$-small set and Jensen's inequality, we obtain
\begin{eqnarray*}
n {\rm d}(G) = 2 e(G) = \sum_{v \in V(G)} \deg(v) = \sum_{i = 1}^t \sum_{v \in A_i} \deg(v)
                                     \le \sum_{i=1}^t (n - |A_i|) |A_i| \le n \left(n - \frac{n}{t} \right).
\end{eqnarray*}
Hence ${\rm d}(G) \le n - \frac{n}{t} = n - \frac{n}{\varphi^{\beta}(G)}$, which is equivalent to $\varphi^{\beta}(G) \ge \frac{n}{n - {\rm d}(G)}$. \\
(ii) Let $V(G) = A_1 \cup A_2 \cup \ldots \cup A_t$ be a partition of $V(G)$ into $t = \varphi^{\alpha}(G)$ $\alpha$-small sets. Then, Corollary \ref{cor_avdeg}(i) and the definition of $\alpha$-small set yield
$$\frac{n}{n- {\rm d(G)}} \le CW(\ov{G}) = \sum_{v \in V(G)} \frac{1}{n - \deg(v)} = \sum_{i=1}^t \sum_{v \in A_i} \frac{1}{n-\deg(v)} \le t = \varphi^{\alpha}(G).$$
 \end{proof}

Let us consider an example. Let $G$ be a graph obtained from $2K_n$ by joining one of the vertices of the first copy of $K_n$ to all the vertices of the second copy of $K_n$. Then $\varphi(G) = 3$, $CW(\ov{G}) = 3 - \frac{2}{n+1}$ and $\varphi^{\beta}(G) = 2$. In this case $\varphi^{\beta}(G) \le CW(\ov{G})$. We do not know if $\varphi^{\beta}(G) \le CW(\ov{G})$ is always true. \\

The inequality chains given in Theorem \ref{thm_ineq-chains}(i) and (ii) together with the fact that $2e(G) = n {\rm d}(G)$ lead to the following corollary.

 \begin{cor}\label{cor_edges_2}
  Let $G$ be a graph on $n$ vertices. Then\\[-5ex]
  \begin{enumerate}
   \item[(i)] $e(G) \le \frac{(\varphi^{\beta}(G)  -1)n^2}{2 \varphi^{\beta}(G)} \le \frac{(\varphi^{\alpha}(G) -1) n^2}{2 \varphi^{\alpha}(G)} \le \frac{(\varphi(G) -1) n^2}{2 \varphi(G)} \le \frac{(\omega(G) -1) n^2}{2 \omega(G)}$;
 \item[(ii)] $e(G) \le \frac{(CW(\ov{G})  -1)n^2}{2 CW(\ov{G})} \le \frac{(\varphi^{\alpha}(G) -1) n^2}{2 \varphi^{\alpha}(G)} \le \frac{(\varphi(G) -1) n^2}{2 \varphi(G)} \le \frac{(\omega(G) -1) n^2}{2 \omega(G)}$.
 \end{enumerate}
 \end{cor}

As remarked for Corollary \ref{cor_edges}, the above bounds on $e(G)$ are better than the bound $e(G)\leq\frac{n^2(\omega(G)-1)}{2\omega(G)}$ from classical Tur\'an's Theorem, because $\omega (G) \geq \varphi(G) \ge \varphi^{\alpha}(G) \ge CW(\ov{G})$ and $\varphi^{\alpha}(G) \ge \varphi^{\beta}(G)$. \\

Analogous to $\alpha$-small and $\beta$-small sets, we can define $\alpha$-large and $\beta$-large sets. Let $G$ be a graph on $n$ vertices and $B$ a subset of $V(G)$. $B$ will be called \emph{$\alpha$-large} if $\sum_{v \in B} \frac{1}{\deg(v)+1} \le 1$ and \emph{$\beta$-large} if ${\rm d}(B) \ge |B| -1$. As for small sets, every large set is an $\alpha$-large set and every $\alpha$-large set is a $\beta$-large set. We also define $\Omega^{\alpha}(G)$ and $\Omega^{\beta}(G)$ as the minimum number of $\alpha$-large sets and, respectively, $\beta$-large sets in which $V(G)$ can be partitioned.

Theorem \ref{thm_ineq-chains} and Corollary \ref{cor_edges} yield, together with the known facts that $\alpha(G) = \omega(\ov{G})$, $\Omega(G) = \varphi(\ov{G})$, $\Omega^{\alpha}(G) = \varphi^{\alpha}(\ov{G})$ and $\Omega^{\beta}(G) = \varphi^{\beta}(\ov{G})$, the following corollaries.\\

  \begin{cor}
 Let $G$ be a graph on $n$ vertices. Then\\[-5ex]
 \begin{enumerate}
 \item[(i)] $\alpha(G) \ge \Omega(G) \ge \Omega^{\alpha}(G) \ge \Omega^{\beta}(G) \ge \frac{n}{{\rm d(G)}+1}$;
 \item[(ii)] $\alpha(G) \ge \Omega(G) \ge \Omega^{\alpha}(G) \ge CW(G) \ge \frac{n}{{\rm d(G)}+1}$;\\
  \end{enumerate}
 \end{cor}

  \begin{cor}
 Let $G$ be a graph on $n$ vertices. Then\\[-5ex]
 \begin{enumerate}
   \item[(i)] $e(G) \ge \frac{n}{2} \left(\frac{n}{\Omega^{\beta}(G)}-1 \right)  \ge \frac{n}{2} \left(\frac{n}{\Omega^{\alpha}(G)}-1 \right)  \ge \frac{n}{2} \left(\frac{n}{\Omega(G)}-1 \right) \ge \frac{n}{2} \left(\frac{n}{\alpha(G)}-1 \right)$;
 \item[(ii)] $e(G) \ge \frac{n}{2} \left(\frac{n}{CW(G)}-1 \right)  \ge \frac{n}{2} \left(\frac{n}{\Omega^{\alpha}(G)}-1 \right)  \ge \frac{n}{2} \left(\frac{n}{\Omega(G)}-1 \right) \ge \frac{n}{2} \left(\frac{n}{\alpha(G)}-1 \right)$.
 \end{enumerate}
 \end{cor}

Let $S^{\alpha}(G)$ and $S^{\beta}(G)$ be the maximum cardinality of an $\alpha$-small set and of a $\beta$-small set of $G$, respectively. Analogously, let Let $L^{\alpha}(G)$ and $L^{\beta}(G)$ be the maximum cardinality of an $\alpha$-large set and of a $\beta$-large set of $G$, respectively. We finish this section with the following theorem.

\begin{thm}
Let $G$ be a graph on $n$ vertices, with maximum degree $\Delta$ and minimum degree $\delta$. Then\\[-5ex]
\begin{enumerate}
\item[(i)] $\alpha(G) \le S_0(G) \le S^\alpha(G) \le S^\beta(G)$\\
 $ \le \left\lfloor \frac{n-\Delta}{2} + \sqrt{\frac{(n-\Delta)^2}{4} + n\Delta - 2 e(G)} \right\rfloor
                                     \le \left\lfloor \frac{1}{2} + \sqrt{ \frac{1}{4} + n^2 - n - 2e(G)} \right\rfloor$;
\item[(ii)] $\omega(G) \le L_0(G) \le L^\alpha(G) \le L^\beta(G)$\\
  $\le \left\lfloor \frac{\delta + 1}{2} + \sqrt{\frac{(\delta+1)^2}{4} - n\delta + 2 e(G)} \right\rfloor
                                 \le \left\lfloor \frac{1}{2} + \sqrt{\frac{1}{4} + 2 e(G)}\right\rfloor$.
\end{enumerate}
\end{thm}

\begin{proof}
The inequality chains $\alpha(G) \le S_0(G) \le S^\alpha(G) \le S^\beta(G)$ and $\omega(G) \le L_0(G) \le L^\alpha(G) \le L^\beta(G)$ follow from Theorem \ref{thm_trivial-bounds}(i) for $k=0$ and Observation \ref{obs_alpha-beta}. The proof of the right side inequalities is analogous to the proof of the Theorem \ref{thm_Hansen} in case $k=0$.
\end{proof}

Note also that Corollary \ref{cor_Hansen} follows from this theorem because of $S_0(G) \le S^{\alpha}(G)$ and $L_0(G) \le L^{\alpha}(G)$.

\end{document}